\documentclass[12pt]{amsart}

\usepackage{amssymb}

\theoremstyle{plain}

\newtheorem{thm}{Theorem}
\newtheorem{cor}[thm]{Corollary}
\newtheorem{lem}[thm]{Lemma}
\newtheorem{prop}[thm]{Proposition}

\theoremstyle{definition}
\newtheorem{defn}{Definition}

\theoremstyle{remark}
\newtheorem{rem}{Remark}

\theoremstyle{remark}

\begin{document}

\title{An algebra involving braids and ties}

\author{F. Aicardi}
\address{ICTP, Trieste(Italy)}
\email{faicardi@ictp.it}
\author{J. Juyumaya}
\address{Instidudo  de  matem\'aticas, Universidad de Valpara\'{\i}so, Gran Breta\~na 1111, Valpara\'{\i}so, Chile
 }
 \email{ juyumaya@uvach.cl}




\begin{abstract}
In this note we  study a  family  of  algebras ${\mathcal E}_n(u)$   with one parameter defined by
generators and relations. The set of generators contains  the  generators of the
usual braids algebra, and another set of  generators  which is  interpreted as ties between
consecutive strings. We also study the representations theory of
the algebra when the parameter is specialized to 1.
\end{abstract}

\maketitle

\section*{Foreword}
This  article  was  written  in  the  2000  year  and  published as  ICTP preprint \cite{AJ}.
The  authors submitted  it  to  a  journal,   but  it  was  rejected  since   according  to the  referee  report
   the  considered  algebra  was the Hecke algebra with another vesture; moreover,
  this algebra should be not related to knot theory.

However, in  the  last  years  the algebras  ${\mathcal E}_n(u)$  was  revisited  by  Ryom-Hansen \cite{rh},  who  found  explicit  bases and  classified  the irreducible  representations.   Moreover,  in  2013 Banjo \cite{eb} determined the complex generic representation theory,  showing that a certain specialization of
this algebra is isomorphic to the small ramified partition algebra.   The  authors of  the  article  have successively proved  that ${\mathcal E}_n(u)$${\mathcal E}_n(u)$   admits  a trace \cite{aijuMMJ1} and    introduced  the  tied  links,  defining   different invariants  for   them \cite{aijuJKTR1}.  More  recently, they  have  considered  Kauffman type  invariants for  tied  links \cite{aijuMZ}  and  have  introduced  a  new algebra (a  generalization of  the BMW  algebra) which  is  related  to ${\mathcal E}_n(u)$. A  Referee  of  \cite{aijuMZ}  suggested to  put in  ArXiv the  present  article.

\section{Introduction}
In this note we continue to study an algebra ${\mathcal E}_n(u)$ of type $A$
 with one parameter $u$ defined  over the fields of rational
functions ${\Bbb C}(u)$. This algebra
was introduced in \cite{ju}, and the definition arose from
 an abstracting procedure of
a non-standard presentation for the  Yokonuma-Hecke algebra (\cite{jk}).
The definition make use of generators and braids relations  between the
generators.

Notice that  the definition of the algebra
${\mathcal E}_n(u)$ that we give  in section 2 includes
less relations than the original definition \cite{ju}. The
definition of our algebra is now using a reduced system of relations.

In  sections 3 and 4, we introduce and study a diagrammatic
interpretation for the generators defining of ${\mathcal E}_n(u)$, which
become useful to obtain information on the linear part of our algebra.

Also, in this note we study the representation theory of the
algebra, whenever the parameter is specialized  to 1. In fact, in
section 5, we construct  families of representations arising
from the hyperoctahedral group. Using our methods, we describe the
theory of representation for low dimension. For instance, we show
the representation theory for ${\mathcal E}_3(1)$.

 \newcount\beziercnt \beziercnt = 100   

 \newcount\grcalca
 \newcount\grcalcb
 \newcount\grcalcc
 \newcount\grcalcd
 \newcount\mordiam
 \newcount\radius
 \newcount\parametr
 \newcount\breit
\newcount\width

 \mordiam = 8   
 \radius = 3    
 \parametr = 0  
 \newcommand\thlines{\thinlines} 

 \newcommand{\braid}{
  \divide \parametr by -3
  \advance \parametr by 10                                 
   \thlines                                                             
   \bezier{\beziercnt}(\parametr,10)(10,7.5)(7,6)
   \bezier{\beziercnt}(0,0)(0,2.5)(3,4)                 
   \bezier{\beziercnt}(0,10)(0,7.5)(5,5)               
 \bezier{\beziercnt}(\parametr,0)(10,2.5)(5,5)
            \parametr = 0 }                       

\newcommand{\zero}[2]{
                    \grcalca = #1
                    \put(\grcalca,#2){\line(0,1){10}}
                   \advance \grcalca by 10
                    \put(\grcalca,#2){\line(0,1){10}}
                    }

\newcommand{\zerr}[2]{
                    \grcalca = #1
                    \advance \grcalca by 10
                    \put(\grcalca,#2){\line(0,1){10}}
                    }

\newcommand{\zerl}[2]{
                    \grcalca = #1
                    \put(\grcalca,#2){\line(0,1){10}}
                   }

\newcommand{\ei}[2]{                    \grcalca = #1
                    \grcalcb = #2
                     \put(\grcalca,\grcalcb){\line(0,1){10}}

                    \advance \grcalca by 10
                    \advance \grcalcb by 5

                    \put(\grcalca,#2){\line(0,1){10}}
                    \put(#1,\grcalcb){\dashbox{1}(10,0)}
                    }

\newcommand{\pzero}[2]{
                    \grcalca = #1
                    \put(\grcalca,#2){\line(0,1){2}}
                   \advance \grcalca by 10
                    \put(\grcalca,#2){\line(0,1){2}}
                    }

\newcommand{\pzerr}[2]{
                    \grcalca = #1
                    \advance \grcalca by 10
                    \put(\grcalca,#2){\line(0,1){2}}
                    }

\newcommand{\pzerl}[2]{
                    \grcalca = #1
                    \put(\grcalca,#2){\line(0,1){2}}
                   }

\newcommand{\pei}[2]{                    \grcalca = #1
                    \grcalcb = #2
                     \put(\grcalca,\grcalcb){\line(0,1){2}}

                    \advance \grcalca by 10
                    \advance \grcalcb by 1

                    \put(\grcalca,#2){\line(0,1){2}}
                    \put(#1,\grcalcb){\dashbox{1}(10,0)}
                    }

 \newcommand{\ibraid}{
 \divide \parametr by -3
  \advance \parametr by 10                         
   \thlines                                        
   \bezier{\beziercnt}(\parametr,10)(10,7.5)(5,5)  
   \bezier{\beziercnt}(0,0)(0,2.5)(5,5)            
   \bezier{\beziercnt}(0,10)(0,7.5)(3,6)           
   \bezier{\beziercnt}(\parametr,0)(10,2.5)(7,4)
     \parametr = 0 }                               

\newcommand{\sibrmor}{\brmor}
\newcommand{\isibrmor}{\ibrmor}
\newcommand{\simor}{\brmor}
\newcommand{\isimor}{\ibrmor}


\newcommand{\nubrmor}{
\thlines                                 
   \bezier{\beziercnt}(9.8,10)(10,8.5)(6,6)  
   \bezier{\beziercnt}(0,0)(0,1.5)(4,4)     
   \bezier{\beziercnt}(0,10)(0,8.5)(3,7)    
   \bezier{\beziercnt}(9.8,0)(10,1.5)(7,3)
\bezier{\beziercnt}(4,4)(4,4)(6,6)
\bezier{\beziercnt}(3,7)(7,7)(7,3)\bezier{\beziercnt}(3,7)(3,3)(7,3)}

\newcommand{\inubrmor}{ 
\thlines
  \bezier{\beziercnt}(4,6)(4,6)(6,4)
   \bezier{\beziercnt}(9.8,10)(10,8.5)(7,7)  
   \bezier{\beziercnt}(0,0)(0,1.5)(3,3)     
   \bezier{\beziercnt}(0,10)(0,8.5)(4,6)    
   \bezier{\beziercnt}(9.8,0)(10,1.5)(6,4)
\bezier{\beziercnt}(3,3)(3,7)(7,7)
\bezier{\beziercnt}(3,3)(7,3)(7,7)}

\newcommand{\rmor}[1]{
\put(0,10){\twist{2}{-4}}
\put(10,10){\twist{-2}{-4}}
\put(5,4){\circle{8}}
\put(5,4){\makebox(0,0){$#1$}}}

 \newcommand{\dualmor}{                      
 \thlines                                    %
 \put(-8,5){\oval(6,6)[l]}                   
 \bezier{\beziercnt}(-.2,10)(0,8)(-4,4)       %
 \bezier{\beziercnt}(-4,4)(-6,2)(-8,2)
\bezier{\beziercnt}(-8,8)(-6,8)(-5,7)        %
 \bezier{\beziercnt}(-1,3)(0,2)(-.2,0)}     %

\newcommand{\dualmors}{                      
 \thlines                                    %
 \put(-8,5){\oval(6,6)[l]}                   
 \bezier{\beziercnt}(-.2,0)(0,2)(-4,6)       %
 \bezier{\beziercnt}(-4,6)(-6,8)(-8,8)
\bezier{\beziercnt}(-8,2)(-6,2)(-5,3)        %
 \bezier{\beziercnt}(-1,7)(0,8)(-.2,10)}     %

 \newcommand{\idualmor}{                      
 \thlines                                     %
 \put(8,5){\oval(6,6)[r]}                     
 \bezier{\beziercnt}(0,10)(0,8)(4,4)           %
 \bezier{\beziercnt}(4,4)(6,2)(8,2)
\bezier{\beziercnt}(8,8)(6,8)(5,7)            %
 \bezier{\beziercnt}(1,3)(0,2)(0,0)}         %

 \newcommand{\idualmors}{                      
 \thlines                                     %
 \put(8,5){\oval(6,6)[r]}                     
 \bezier{\beziercnt}(0,0)(0,2)(4,6)           %
 \bezier{\beziercnt}(4,6)(6,8)(8,8)
\bezier{\beziercnt}(8,2)(6,2)(5,3)            %
 \bezier{\beziercnt}(1,7)(0,8)(0,10)}         %

\newcommand{\dualbrmors}{\dualmors\put(-3,5){\circle{6}}}
\newcommand{\idualbrmors}{\idualmors\put(3,5){\circle{6}}}
\newcommand{\dualbrmor}{\dualmor\put(-3,5){\circle{6}}}
\newcommand{\idualbrmor}{\idualmor\put(3,5){\circle{6}}}

 \newcommand{\mor}[2]{
   \thlines
   \grcalca = #2
   \advance \grcalca by -\radius
 \advance \grcalca by -\radius
   \divide \grcalca by 2                               
   \put(0,#2) {\line(0,-1){\grcalca}}             
   \put(0,0) {\line(0,1){\grcalca}}                
   \grcalca = #2                                        
   \divide \grcalca by 2                           
   \put(0,\grcalca) {\circle{\mordiam}}     
   \put(0,\grcalca) {\makebox(0,0){$#1$}} } 

\newcommand{\umor}[2]{                          %
 \thlines                                                         
 \grcalca = #2                                                
\advance \grcalca by -\radius
 \advance \grcalca by -\radius
\put(0,\grcalca) {\line(0,-1){\grcalca}}       
\advance \grcalca by \radius
\put(0,\grcalca) {\circle{\mordiam}}
\put(0,\grcalca) {\makebox(0,0){$#1$}} }        
\newcommand{\coumor}[2]{                          %
 \thlines                                                          
 \grcalca = #2                                                  
     \advance \grcalca by -\radius
 \advance \grcalca by -\radius
\put(0,6) {\line(0,1){\grcalca}}      
\put(0,\radius) {\circle{\mordiam}}             
\put(0,\radius) {\makebox(0,0){$#1$}} }         

  \newcommand{\multmor}[3]{                    
    \thlines                                   
    \grcalca = #2
    \grcalcb = #3                              
    \advance \grcalca by 5                     
    \put(-2.5,0){\framebox(\grcalca,\grcalcb)  
     {$#1$}}}

 \newcommand{\twist}[2]{                     
   \thlines                                               
   \grcalca = #1                                      
   \divide \grcalca by 2                          
   \grcalcb = #2                                      
   \divide \grcalcb by 2                          
   \grcalcc = #2                                      
   \divide \grcalcc by 3                          
   \bezier{\beziercnt}(0,0)(0,\grcalcc)(\grcalca,\grcalcb)
   \grcalcd = #2
   \advance \grcalcd by -\grcalcc
   \multiply \grcalcc by 3
   \bezier{\beziercnt}(\grcalca,\grcalcb)(#1,\grcalcd)(#1,#2) }

 \newcommand{\idgr}[1]{                     
   \thlines                                 
   \put(0,0) {\line(0,1){#1}}}              

  \newcommand{\rigco}[1]{                   
    \thlines                                
    \put(0,0) {\line(1,0){#1}}              
    \put(#1,0) {\line(-5,2){#1}}}           

\newcommand{\ig}[1]{\idgr{#1}}

  \newcommand{\lefco}[1]{                   
    \thlines                                
    \put(#1,0) {\line(0,1){5}}              
    \put(0,0) {\line(1,0){#1}}              
    \put(0,0) {\line(5,2){#1}}}             

  \newcommand{\rigmu}[1]{                   
    \thlines                                
    \put(0,0) {\line(0,1){5}}               
    \put(0,5) {\line(1,0){#1}}              
    \put(#1,5) {\line(-5,-2){#1}}}          

  \newcommand{\lefmu}[1]{                   
    \thlines                                
    \put(#1,0) {\line(0,1){5}}              
    \put(0,5) {\line(1,0){#1}}              
    \put(0,5) {\line(5,-2){#1}}}            

\newcommand{\objo}[1]{
   \put(0,3){\makebox(0,0)[b]{\mbox{$#1$}}}}

 \newcommand{\oo}{\objo}

 \newcommand{\obju}[1]{
   \put(0,-5){\makebox(0,0)[b]{\mbox{$#1$}}}} 
\newcommand{\ou}{\obju}

\newcommand{\bgr}[3]{
   \unitlength=#1 
\divide \unitlength by 2
\advance  \unitlength by #1
 \grcalca = #3
   \advance \grcalca by 6
 \breit = #2
\advance \breit  by 5
   \begin{picture}(#2,\grcalca)}

  \def\eeee{\put(\breit,0){\obju{\mbox{\normalsize\rule{0.75ex}{1.5ex}}}}}
\thlines

\newcommand{\egr}{\end{picture}}

\newcommand{\fbgr}[2]{\scriptsize\bgr{.4ex}{#1}{#2}}
\newcommand{\fegr}{\egr\nml
}
\newcommand{\cbgr}[2]{\scriptsize\begin{center}\bgr{.4ex}{#1}{#2}}
\newcommand{\cegr}{\egr\end{center}\nml\vspace{1ex}
}

 \newcommand{\bsgr}[1]{
   \newsavebox{#1}
   \savebox{#1}(0,0)[bl]}

\section{Generators}
\begin{defn} Let  $n$ be a natural number.  The
algebra ${\mathcal E}_n(u)$ is defined as the associative
algebra over the field of rational function ${\Bbb C} (u)$
with generators
$$
1,T_1, \ldots ,T_{n-1}, E_1,\ldots ,E_{n-1}
$$
subject to the following relations:

\begin{equation}\label{eq1}
T_iT_j=T_jT_i  \quad if \quad \vert i-j \vert >1
\end{equation}
\begin{equation}\label{eq2}
T_iT_jT_i=T_jT_iT_j \quad  if \quad \vert i-j \vert =1
\end{equation}
\begin{equation}\label{eq3}
E_i^2=E_i
\end{equation}
\begin{equation}\label{eq4}
E_iE_j=E_jE_i  \quad  \forall  i,j
\end{equation}
\begin{equation}\label{eq5}
 E_iT_i=T_iE_i
\end{equation}
\begin{equation}\label{eq6}
E_iT_j=T_jE_i  \quad if \quad \vert i-j\vert >1
\end{equation}
\begin{equation}\label{eq7}
 E_jT_iT_j=T_iT_jE_i  \quad if \quad |i-j|=1
\end{equation}
\begin{equation}\label{eq8}
E_iE_jT_j=E_iT_jE_i=T_jE_iE_j  \quad if \quad\vert i-j\vert =1
\end{equation}
\begin{equation}\label{eq9}
T_i^2=1 + (u^{-1}-1)E_i(1-T_i)
\end{equation}
\end{defn}
The original definition of the algebra ${\mathcal E}_n(u)$ given in \cite{ju}
contained 2 superfluous
 relations with respect to the above definition. Throughout  of the
diagrammatical interpretation of the generators,
in terms of braids and ties, it was easy to remove such relations.


As in Lemma 3.1 in \cite{bw} one can  prove that the algebras
${\mathcal E}_n(u)$ are finite dimensional. In fact, firstly note
that any word of ${\mathcal E}_2(u)$ is a linear combination of words
 in  $1$, $T_1$, $E_1$ and $T_1E_1$. Now, suppose that any words of
${\mathcal E}_n(u)$ in
$1$, $T_1$, $\ldots$, $T_{n-1}$,  $E_1$, $\ldots$, $E_{n-1}$ can be written
as a linear combination of  words having at most one
$T_{n-1}$, $E_{n-1}$,  or $T_{n-1}E_{n-1}$.
Then, using the defining relations of ${\mathcal E}_n(u)$  it is not hard
 to verify that any word below is a linear combination of words
having at most one $T_n$, $E_n$ or $E_nT_n$:

$$
\begin{array}{lll}
  T_nR_{n-1}T_n\quad  & T_nR_{n-1}T_nE_n\quad & T_nR_{n-1}E_n \\
T_nE_nR_{n-1}T_n\quad   &  T_nE_nR_{n-1}T_nE_n\quad & T_nE_nR_{n-1}E_n \\
E_nR_{n-1}T_n \quad& E_nR_{n-1}T_nE_n\quad &  E_nR_{n-1}E_n,
\end{array}
$$

where $R_{n-1}\in \{1, T_{n-1}, E_{n-1}, T_{n-1}E_{n-1}\}$. Then
 from an inductive argument one can  deduce the following proposition.

\begin{prop}\label{finite}
In the algebra ${\mathcal E}_n(u)$ any word in
$1$, $T_1$, $\ldots$, $T_{n-1}$,  $E_1$, $\ldots$, $E_{n-1}$ is a
linear combination of words in $T_i$'s, $E_i$'s having at most one
$R_{n-1}$, where $R_{n-1}$ $\in$ $\{T_{n-1}$, $E_{n-1}$, $T_{n-1}E_{n-1}\}$.
Hence  ${\mathcal E}_n(u)$ is finite dimensional.
\end{prop}

Now, as in  the Iwahori-Hecke algebra we can take a system of linear
generates  for ${\mathcal E}_{n}(u)$ (which in our case will be redundant)
in the following way: we define $U_1 = \{1, T_1, E_1, T_1E_1\}$, and
$U_i$   by

$$
U_i := \{1\} \cup T_i U_{i-1} \cup E_i U_{i-1} \cup T_iE_i U_{i-1} \qquad
(2\leq i \leq n).
$$

Using induction and Proposition ~\ref{finite} we deduce that
 ${\mathcal E}_n(u)$ is generated linearly by all the products of the form
$u_1u_2\cdots u_{n-1}$, where $u_i\in U_i$. From where we deduce
\begin{equation}\label{Yes}
{\mathcal E}_{n+1}(u)=   \sum_{1\leq i\leq n} Y_iY_{i+1}\cdots
Y_{n}{\mathcal E}_{n}(u) + {\mathcal E}_{n}(u),
\end{equation}
where $Y_j \in \{T_j, E_j, T_jE_j\}$.

\begin{cor}\label{span}
A basis for ${\mathcal E}_2(u)$ is $\{1, T_1, E_1, T_1E_1\}$. And
${\mathcal E}_3(u)$ is spanned linearly by:
$$
L, LE_1, LE_2, LE_1E_2, LE_2T_1,
$$
where $L \in \{1, T_1, T_2, T_1T_2, T_2T_1, T_1T_2T_1\}$.
\end{cor}

\begin{proof}
The proof follows directly from (\ref{Yes}) and the lemma below.
\end{proof}

\begin{lem}\label{lema} For all $i, j$ such that $\vert i-j\vert =1$,
we have:

\noindent{ \bf ~\ref{lema}.1}
$
E_jT_i = T_iT_jE_iT_j  +(u-1)( T_iT_jE_iE_j + T_iE_iE_j)
$

\noindent{ \bf ~\ref{lema}.2 }
$
T_jE_jT_i = T_iT_jT_iE_iT_j  +(u-1)( T_iT_jT_iE_iE_j + T_jT_iE_iE_j).
$
\end{lem}
\begin{proof}
 From (\ref{eq7}) we get
\begin{eqnarray*}
(T_iT_jE_i)T_j & = & E_jT_iT_j^2\\
                & = & E_jT_i (1 +  (u-1)(E_j - E_jT_j)) \qquad \mbox
        {\rm (from (\ref{eq9}))}\\
        & = & E_jT_i + (u-1)(T_iE_iE_j - T_iT_jE_iE_j)
        \qquad \mbox{\rm (from (\ref{eq8}))}.
\end{eqnarray*}
Thus the assertion (\ref{lema}.1) follows.

Multiplying  ~\ref{lema}.1  on the left by $T_j$, and   using
(\ref{eq2}) we get  ~\ref{lema}.2 .
\end{proof}
 
\section{A diagram representation}
In this section we define an isomorphism $\varphi$ from the algebra
${\mathcal E}_n(u)$ to the algebra ${\mathcal B}{\sf H}_n(u)$ of
special diagrams. The diagrams corresponding to the words of
${\mathcal E}_n(u)$
 look like  the standard diagrams of the elements of the braid group
$B_{n}$ (braids with $n$ strings), with a new structure between
adjacent strings.

\subsection{Word's representation}
We first introduce a one-to-one correspondence $\varphi$ between
the free algebra ${\mathcal A}$ over ${\Bbb C}(u)$ generated by
$1, T_i, E_i$
 and the free algebra ${\mathcal B}$ over ${\Bbb C}(u)$ generated by
diagrams-generators as follows.

\begin{defn}  The image by $\varphi$ in ${\mathcal B}$
of a generator ($1,T_i,E_i)$ of ${\mathcal A}$ is called {\em row}
 \end{defn}

Let $\varphi(1)$ be the row-diagram

\setlength\unitlength{0.4ex}
\begin{picture}(100,35)(-40,5)

\bgr{0.4ex}{100}{20}

\zerl{10}{10}
\zerl{20}{10}
\zerl{30}{10}
\zerl{50}{10}
\zero{60}{10}
\put(38,15){$\ldots$}
\scriptsize
\put(10,21){$1$}
\put(20,21){$2$}

\put(30,21){$3$}
\put(48,21){$n-2$}
\put(58,21){$n-1$}
\put(70,21){$n$}
\egr
\end{picture}

Let $\varphi(T_i)$ be the row-diagram

\begin{picture}(100,35)(-40,5)

\bgr{0.4ex}{100}{20}

\zero{10}{10}
\put(40,10)\ibraid

\zerl{70}{10}
\put(28,15){$\ldots$}
\put(58,15){$\ldots$}

\scriptsize

\put(10,21){$1$}
\put(20,21){$2$}
\put(40,21){$i$}
\put(48,21){$i+1$}
\put(70,21){$n$}
\egr
\end{picture}

Let $\varphi(E_i)$ be the row-diagram

\begin{picture}(100,35)(-40,5)

\bgr{0.4ex}{100}{20}

\zero{10}{10}
\ei{40}{10}

\zerl{70}{10}
\put(28,15){$\ldots$}
\put(58,15){$\ldots$}

\scriptsize
\put(10,21){$1$}
\put(20,21){$2$}
\put(40,21){$i$}
\put(48,21){$i+1$}
\put(70,21){$n$}
\egr
\end{picture}

The application $\varphi$ sends the multiplication $AB$ of two words $A$
and $B$
 (in particular of two generators) to the diagram
obtained putting the diagram $\varphi(B)$ below $\varphi(A)$.

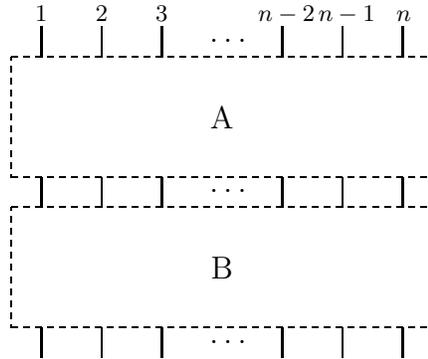
\setlength\unitlength {.8 cm}
\begin{picture}(10,7)(-5.5,.5)


\put (1,1){\line(0,1){.5}}
\put (2,1){\line(0,1){.5}}
\put (3,1){\line(0,1){.5}}
\put (5,1){\line(0,1){.5}}
\put (6,1){\line(0,1){.5}}
\put (7,1){\line(0,1){.5}}

\put(3.8,1.25){$\ldots$}
\put(.5,1.5){\dashbox{.1}(7,2){B}}

\put (1,3.5){\line(0,1){.5}}
\put (2,3.5){\line(0,1){.5}}
\put (3,3.5){\line(0,1){.5}}
\put (5,3.5){\line(0,1){.5}}
\put (6,3.5){\line(0,1){.5}}
\put (7,3.5){\line(0,1){.5}}

 \put(3.8,3.75){$\ldots$}
\put(.5,4){\dashbox{.1}(7,2){A}}
\put (1,6){\line(0,1){.5}}
\put (2,6){\line(0,1){.5}}
\put (3,6){\line(0,1){.5}}
\put (5,6){\line(0,1){.5}}
\put (6,6){\line(0,1){.5}}
\put (7,6){\line(0,1){.5}}

\put(3.8,6.25){$\ldots$}

\scriptsize
\put(.9,6.6){$1$}
\put(1.9,6.6){$2$}

\put(2.9,6.6){$3$}
\put(4.6,6.6){$n-2$}
\put(5.6,6.6){$n-1$}
\put(6.9,6.6){$n$}

\end{picture}

A linear combination of words with coefficients in ${\mathbb C}(u)$
is sent by $\varphi$ in the same formal combinations of diagrams corresponding
to the words with the same coefficients.

\begin{prop}\label{corres}
The above construction defines a one to one correspondence between
 ${\mathcal A}$ and  ${\mathcal B}$.
 \end{prop}

\begin{proof}
It  suffices to see that the inverse map associates to every
diagram one and only one word. In fact, this word is obtained
reading the diagram from top to bottom, and writing in the order
the $\varphi^{-1}$ of the encountered  rows.
\end{proof}

\subsection{Definition of the algebra ${\mathcal B}{\sf H}_n(u)$. }
We  translate  into the diagram language the set
of relations ~\ref{eq1} to ~\ref{eq9}. Adding these relations to the free
algebra ${\mathcal B}$ we obtain a new algebra ${\mathcal B}{\sf H}_n(u)$.

Note that the first eight relations
involve only single words.
 Because of Proposition ~\ref{corres} we will use the same symbols for
elements of
${\mathcal E}_n(u)$  and their images ${\mathcal B}{\sf H}_n(u)$.

The following relations between  $T_i$ are the standard
relations of the braid groups.
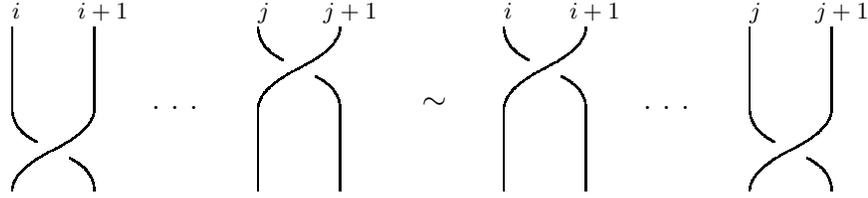
\begin{equation}\label{re1}
T_jT_i=T_iT_j,  \quad {\rm  if} \quad  \vert i-j\vert >1.
\end{equation}

\setlength\unitlength{0.4ex}
\begin{picture}(100,40)(-10,5)

\bgr{0.4ex}{10}{23}

\put(10,5){\ibraid}
\put(40,5){\twist{0}{10}}
\put(50,5){\twist{0}{10}}
\put(27,15){$.\ .\ .$}
\put(10,15){\twist{0}{10}}
\put(20,15){\twist{0}{10}}
\put(40,15){\ibraid}

\put(60,15){$\sim$}

\put(70,15){\ibraid}
\put(100,15){\twist{0}{10}}
\put(110,15){\twist{0}{10}}
\put(87,15){$.\ .\ .$}
\put(70,5){\twist{0}{10}}
\put(80,5){\twist{0}{10}}
\put(100,5){\ibraid}

 \scriptsize
\put(10,26){$i$}
\put(18,26){$i+1$}
\put(40,26){$j$}
\put(48,26){$j+1$}
\put(70,26){$i$}
\put(78,26){$i+1$}
\put(100,26){$j$}
\put(108,26){$j+1$}
 \egr

\end{picture}

\begin{equation}\label{re2}
T_iT_jT_i=T_jT_iT_j, \quad {\rm  if} \quad  \vert i-j\vert =1.
\end{equation}

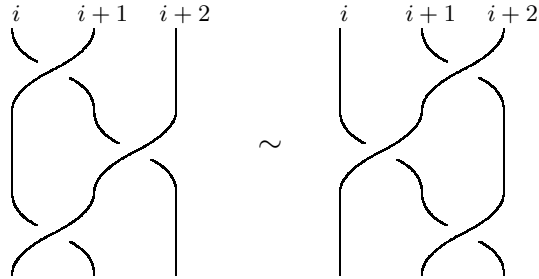
\begin{picture}(100,60)(-40,5)

\bgr{0.4ex}{100}{38}

\put(10,10){\ibraid}
\put(30,10){\twist{0}{10}}
\put(10,20){\twist{0}{10}}
\put(20,20){\ibraid}
\put(30,30){\twist{0}{10}}
\put(10,30){\ibraid}

\put(40,25){$\sim$}

\put(60,10){\ibraid}
\put(50,10){\twist{0}{10}}
\put(70,20){\twist{0}{10}}
\put(50,20){\ibraid}
\put(50,30){\twist{0}{10}}
\put(60,30){\ibraid}

\scriptsize
\put(10,41){$i$}
\put(18,41){$i+1$}
\put(28,41){$i+2$}
\put(50,41){$i$}
\put(58,41){$i+1$}
\put(68,41){$i+2$}
\egr
\end{picture}

The $E_i$ generators commute one another, and any (natural) power of $E_i$
coincides with $E_i$.
\begin{equation}\label{re3}
E_i^2=E_i.
\end{equation}

\begin{picture}(100,25)(-55,5)

\bgr{0.4ex}{10}{18}

\ei{10}{0}

\ei {10}{10}

\put(28,10){$\sim$}

\ei {40}{5}
\zero{40}{10}
\zero{40}{0}

 \egr
\end{picture}

\begin{equation}\label{re4}
E_jE_i=E_iE_j,  \quad \forall   i,j.
\end{equation}

\begin{picture}(100,40)(-10,5)
\bgr{0.4ex}{10}{23}

\ei{10}{5}
\zero{40}{5}
\put(27,15){$.\ .\ .$}
\zero{10}{15}

\ei {40}{15}

\put(60,15){$\sim$}

\ei {70}{15}
\zero{100}{15}

\put(87,15){$.\ .\ .$}
\zero{70}{5}
\ei {100}{5}

\scriptsize
\put(10,26){$i$}
\put(18,26){$i+1$}
\put(40,26){$j$}
\put(48,26){$j+1$}
\put(70,26){$i$}
\put(78,26){$i+1$}
\put(100,26){$j$}
\put(108,26){$j+1$}
\egr
\end{picture}

Relations involving $E_i$'s  and $T_i$'s generators:
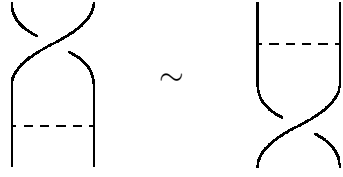
\begin{equation}\label{re5}
 E_iT_i=T_iE_i.
\end{equation}

\begin{picture}(100,40)(-55,5)
\bgr{0.4ex}{10}{18}

\ei{10}{5}

\put(10,15){\ibraid}

\put(28,15){$\sim$}

\put(40,5){\ibraid}

\ei{40}{15}

 \egr
\end{picture}

\begin{equation}\label{re6}
E_jT_i=T_iE_j,  \quad {\rm  if} \quad \vert i-j\vert >1.
\end{equation}


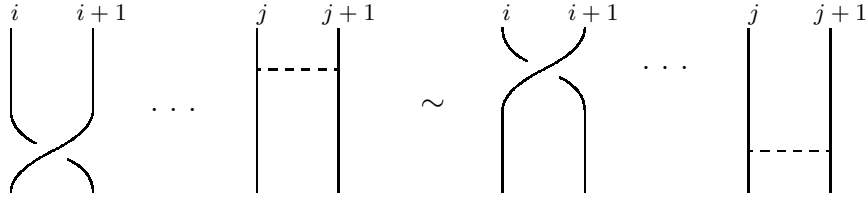
\begin{picture}(100,45)(-5,0)

\bgr{0.4ex}{10}{23}

\put(10,5){\ibraid}
\zero{10}{15}

\put(27,15){$.\ .\ .$}
\zero{40}{5}
\ei {40}{15}

\put(60,15){$\sim$}

\put(70,15){\ibraid}
\zero{100}{15}

\put(87,20){$.\ .\ .$}
\zero{70}{5}
\ei {100}{5}

\scriptsize
\put(10,26){$i$}
\put(18,26){$i+1$}
\put(40,26){$j$}
\put(48,26){$j+1$}
\put(70,26){$i$}
\put(78,26){$i+1$}
\put(100,26){$j$}
\put(108,26){$j+1$}
\egr
\end{picture}

\begin{equation}\label{re7}
 E_jT_iT_j=T_iT_jE_i,  \quad {\rm  if} \quad  \vert i-j\vert =1.
\end{equation}

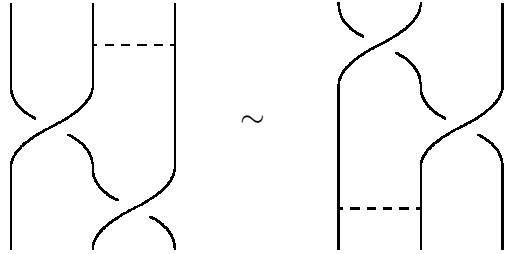
\begin{picture}(100,45)(-40,5)
\bgr{0.4ex}{10}{30}

\ei{20}{25}
\zerl{10}{25}
\put(10,15){\ibraid}
\zerr{20}{15}
\put(20,5){\ibraid}
\zerl{10}{5}

\put(38,20){$\sim$}

\ei{50}{5}
\zerr{60}{5}
\put(60,15){\ibraid}
\zerl{50}{15}
\put(50,25){\ibraid}
\zerr{60}{25}

 \egr
 \end{picture}

\begin{equation}\label{re8}
E_iE_jT_j=E_iT_jE_i=T_jE_iE_j,  \quad {\rm  if} \quad  \vert i-j \vert =1.
\end{equation}

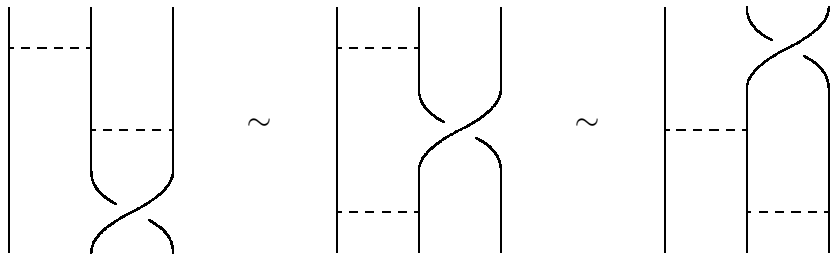
\begin{picture}(100,44)(-30,5)
\bgr{0.4ex}{10}{30}

\ei{0}{25}
\ei{10}{15}
\zerl{0}{15}
\put(10,5){\ibraid}
\zerl{0}{5}
\zero{10}{25}



\put(29,20){$\sim$}

\ei{40}{25}
\zero{50}{25}
\zerl{40}{15}
\put(50,15){\ibraid}
\zerr{50}{5}
\ei{40}{5}



\put(69,20){$\sim$}

\zerl{80}{25}
\zero{80}{5}
\zero{90}{15}
\put(90,25){\ibraid}
\ei{80}{15}
\ei{90}{5}

 \egr
\end{picture}

The last relation, corresponding to (\ref{eq9}), is of type \lq\lq Skein
rule\rq\rq. By this relation a single diagram is equivalent to a
linear combination of diagrams. By eq. (\ref{eq9})  and  relation $T_iT_i^{-1}=1$ one obtains the following  expression  for
by $T_i^{-1}$:
\begin{equation}\label{re9}
 T_i^{-1}=T_i + (u-1)E_iT_i+(1-u)E_i.
\end{equation}

Hence a diagram containing an inverse generator $T_i^{-1}$ is equivalent to the
following combination of diagrams containing $T_i$, $E_iT_i$
and $E_i$ in place of $T_i^{-1}$:

\begin{picture}(100,40)(-15,5)
\bgr{0.4ex}{10}{23}

\put(10,10){\braid}
\put(24,15){$\sim$}
\put(30,10){\ibraid}

\put(42,15){$+\quad (u-1)$}

\put(65,5){\ibraid}
\ei{65}{15}
\put(77,15){$+\quad (1-u)$}

\ei{100}{10}

 \egr
\end{picture}

\begin{prop}
The diagram relations ~\ref{re1} to ~\ref{re9} define the algebra
${\mathcal B}{\sf H}_n(u)$. This algebra is naturally isomorphic
to ${\mathcal E}_n(u)$ by construction.
 \end{prop}

\section{Diagram's reduction}
The representation of the elements of an abstract algebra by some
geometrical objects (which may be taken as  models of some physical objects) is
interesting if the equivalence relations of the algebra become, in
terms of these objects, natural moves which leave unaltered some
property of them. Thus the geometrical intuition can help to find
immediately equivalence relations between objects.

In this section an interpretation of the element $E_i$ is given so
that the allowed moves involving $E_i$'s relate  equivalent objects
with respect to the algebra.

\subsection{Useful relations} We need some relations coming from
 ~\ref{eq1} to ~\ref{eq9}.
\begin{prop} \label{tim}
The relations ~\ref{eq5} to ~\ref{eq8} are true substituting $T_i$ for
$T_i^{-1}$.
\end{prop}
\begin{proof}
If we substitute  $T_i^{-1}$ into relations ~\ref{eq5} to ~\ref{eq8} given
by relation ~\ref{re9} we find identities by means of equations 
~\ref{eq3} to ~\ref{eq8}.
\end{proof}
Multiplying the terms of equations ~\ref{eq7} to  the left by $T_i^{-1}$ and
 to the right by $T_j^{-1}$ we obtain
\begin{equation}\label{re13}
T_i^{-1}E_jT_i=T_j E_i T_j^{-1}.
\end{equation}

Now, we have
\begin{eqnarray*}
T_i^{-1}E_j T_i & = & T_i^{-1}E_j T_iT_i^{-1}E_j T_i\\
                & = & T_i^{-1}E_j T_iT_j E_i T_j^{-1}\qquad (\mbox{from}\quad ~\ref{re13}) \\
        & = & T_i^{-1}  T_iT_jT_j^{-1}E_iE_iT_j^{-1} \qquad (\mbox{from}\quad ~\ref{eq7}) .
\end{eqnarray*}
Then
\begin{equation}\label{re14}
T_i^{-1}E_j T_i = T_j^{-1}E_i  T_j,
\end{equation}
or equivalently
\begin{equation}\label{re14A}
u(T_iE_jT_i - T_jE_iT_j) = (u-1)(E_jT_iE_j - E_iT_jE_i).
\end{equation}

Finally we consider   relation ~\ref{eq2}. We multiply it by $T_i^{-1}E_j T_i$
and by the preceding relation we obtain
$$
T_iT_jT_iT_i^{-1}E_j T_i=T_jT_iT_jT_j^{-1}E_i T_j,
$$
i.e.
\begin{equation}\label{re15}
T_iT_jE_j T_i=T_jT_iE_i T_j.
\end{equation}
\subsection{$E_i$'s as ties}
The  element $E_i$, represented in the diagrams by a dashed line
between the strings $i$ and $i+1$, can be interpreted as a rigid
bar free to move up and down between two adjacent strings as far
as they remain at the same distance. This is the meaning of relations
 ~\ref{re4}, ~\ref{re6}, ~\ref{re7} and  ~\ref{re5},
if we interpret element $T_i$ as a twist in
the 3-space.
The non commutativity of $E_i$ with $T_j$ when $\vert i-j\vert =1$ is thus
interpreted as the obstacle to movements of the bar.

Consider now the following diagram relations whose validity we showed in
the preceding subsection:
The first one

\begin{equation}
 E_jT_i^{-1}T_j=T_i^{-1}T_jE_i,   \quad {\rm  if} \quad  \vert i-j \vert =1.
\end{equation}

\begin{picture}(100,45)(-30,5)
\bgr{0.4ex}{10}{30}

\ei{20}{25}
\zerl{10}{25}
\put(10,15){\braid}
\zerr{20}{15}
\put(20,5){\ibraid}
\zerl{10}{5}

\put(38,20){$\sim$}

\ei{50}{5}
\zerr{60}{5}
\put(60,15){\ibraid}
\zerl{50}{15}
\put(50,25){\braid}
\zerr{60}{25}

 \egr
 \end{picture}

\noindent shows that the bar can slide passing through the strings.
The following relations still indicate  that in a rigid horizontal move of the
bar, its bypassing of a string is allowed. Note that the corresponding moves
generalize combinations of  Reidemeister moves of second type.
\begin{equation}
 T_i^{-1}E_jT_i=T_iE_jT_i^{-1}= T_j^{-1}E_iT_j=T_jE_iT_j^{-1},  \qquad  (\vert i-j \vert =1).
\end{equation}

\begin{picture}(100,40)(-10,5)
\bgr{0.4ex}{10}{26}

\put(0,5){\ibraid}
\zerr{10}{5}

\pzero{0}{15}
\pei {10}{15}

\put(0,17){\braid}
\zerr{10}{17}

\put(25,16){$\sim$}

\put(33,5){\braid}
\zerr{43}{5}

\pzero{33}{15}
\pei {43}{15}

\put(33,17){\ibraid}
\zerr{43}{17}

\put(58,16){$\sim$}

\put(76,5){\ibraid}
\zerl{66}{5}

\pzero{76}{15}
\pei {66}{15}

\put(76,17){\braid}
\zerl{66}{17}

\put(92,16){$\sim$}

\put(110,5){\braid}
\zerl{100}{5}

\pzero{110}{15}
\pei {100}{15}

\put(110,17){\ibraid}
\zerl{100}{17}

 \egr
\end{picture}

Finally,  equation ~\ref{re15}
\begin{equation}
T_iT_jE_j T_i=T_jT_iE_i T_j,  \qquad (\vert i-j\vert =1).
\end{equation}

\begin{picture}(100,60)(-30,5)
\bgr{0.4ex}{10}{40}

\put(10,10){\ibraid}
\put(30,10){\twist{0}{10}}
\put(10,20){\twist{0}{10}}
\put(20,20){\ibraid}
\pei{20}{30}
\pzerl{10}{30}
\put(30,32){\twist{0}{10}}
\put(10,32){\ibraid}

\put(40,25){$\sim$}

\put(60,10){\ibraid}
\put(50,10){\twist{0}{10}}
\put(70,20){\twist{0}{10}}
\put(50,20){\ibraid}

\pei{50}{30}
\pzerr{60}{30}
\put(50,32){\twist{0}{10}}
\put(60,32){\ibraid}

 \egr
\end{picture}

\noindent shows that also in the generalized third Reidemeister move the bar is allowed
to bypass the string.

We remark that relation ~\ref{re8} is the only exception to the rule
of sliding of the bar representing an $E_i$. We can interpret  this
relation as a sort of cooperation between two adjacent bars:
$E_i$ do not commute with $T_{i+1}$, but commutes with
$E_{i+1}T_{i+1}$.

\subsection{Equivalent relations of Skein type}
To conclude, we add some useful relations equivalent to
 ~\ref{eq9}.

Multiplying relation ~\ref{eq9}  by  $E_i$  we get the following
relations:
\begin{equation}
u E_iT_i -E_iT_i^{-1}=(u-1)E_i.
\end{equation}

\begin{picture}(100,20)(-20,15)
\bgr{0.4ex}{10}{15}

\put(5,15){$u$}
\put(10,10){\ibraid}
\pei{10}{20}
\put(24,15){$-$}
\put(30,10){\braid}
\pei{30}{20}

\put(50,15){$\sim \;\;\;\;\;(u-1)$}

\zero {75}{10}
\pei{75}{20}

 \egr
 \end{picture}

\begin{equation}
  E_iT_i - E_iT_i^{-1}=T_i - T_i^{-1}.
\end{equation}

\begin{picture}(100,30)(-20,5)
\bgr{0.4ex}{10}{15}

\put(10,10){\ibraid}
\pei{10}{20}
\put(24,15){$-$}
\put(30,10){\braid}
\pei{30}{20}

\put(48,15){$\sim$}

\put(60,10){\ibraid}

\put(74,15){$-$}
\put(80,10){\braid}

 \egr
\end{picture}

\section{Representation theory of
${\mathcal E}_n(1)$}
In this section we study the representation theory of
${\mathcal E} := {\mathcal E}_n(1)$.

We will use the standard definition of partition, and as usual we will
regard the partition as Young diagram. Also we will denote again by
$\alpha$ the irreducible
representation of ${\mathcal S}_n$ associated to a partition
$\alpha$ of $n$.

The symmetric group ${\mathcal S}_n$ is regarded as a Coxeter group generated
by the set of  elemental   transpositions $S$, that is, $S=\{s_1, \ldots ,
s_{n-1}\}$, $s_i=(i, i+1)$.

We are going to show a series of irreducible representations of ${\mathcal E}$
that come from the hyperoctahedral group $W_n$.  In other words $W_n$ is
the wreath product $C_2\wr {\mathcal S}_n$, where $C_2$ is the group with
two elements, say, $C_2=\{1,t\}$. Thus, $W_n$ has a  presentation defined by
the following Dynkin diagram
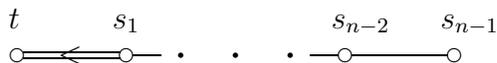
\begin{center}
\setlength\unitlength{0.2ex}
\begin{picture}(350,40)
\put(82,20){$t$}
\put(120,20){$s _{1}$}
\put(200,20){$s_{n-2}$}
\put(240,20){$s_{n-1}$}

\put(85,10){\circle{5}}
\put(87.5,11){\line(1,0){35}}
\put(87.5,9){\line(1,0){35}}
\put(125,10){\circle{5}}
\put(127.5,10){\line(1,0){10}}

\put(145,10){\circle*{2}}
\put(165,10){\circle*{2}}
\put(185,10){\circle*{2}}

\put(205,10){\circle{5}}
\put(207.5,10){\line(1,0){35}}
\put(245,10){\circle{5}}
\put(192.5,10){\line(1,0){10}}

\put(100,7){$<$}

\end{picture}
\end{center}

We have $W_n=C\rtimes {\mathcal S}_n$, where $C = C_2\times \cdots
\times C_2$ ($n$-times).

We define the elements $t_i\in C$ inductively as $t_1 = t$,
$t_{i+1} =s_i t_is_i$.

We shall represent $W_n$ as the subgroup of the monomials matrices of
$GL_n$, whose  non-zero entries  are 1 or $t$. Then $C$ is the
diagonal subgroup of $GL$ with 1's and $t$'s on the diagonal; and
$t_i$ is the diagonal matrix with $t$ in the position $(i, i)$,
and 1 otherwise.

Let $a, b \geq 0$ such that $a + b= n$. Let $W_{(a,b)}$ be the subgroup of
$W_n$ isomorphic to the group $W_a \times W_b$, according to the diagonal
embedding $GL_a \times GL_b \subseteq GL_n$.

In order to describe the representation theory of $W_n$, we consider
 the following linear character, $\epsilon :W_n \longrightarrow\{-1, 1\}$,
defined by $\epsilon(t)= -1$, $\epsilon (s_i)= 1$. We consider,
using the natural
surjection $\pi$ of $W_m$ onto ${\mathcal S}_m$ and pull-back
mechanism, a representation of ${\mathcal S}_m$  as a representation of
$W_m$.

\begin{thm} (Specht, see \cite{gp})
The types of irreducible representations of $W_n$ are para\-me\-tri\-zed
by the bipartitions of $n$. Moreover, all irreducible representations of $W_n$
can be underlying like an induced representation of the form
$$
V_{(\alpha, \beta)}={\mbox Ind}_{W_{(a,b)}}^{W_n}\alpha \otimes \epsilon\beta,
$$
where $\alpha$ is a partition of $a$, and $\beta$ is a partition of $b$,
$a + b =n$.
\end{thm}

The key point for taking representations of ${\mathcal E}$ from the
 group $W_n$ is the following  proposition.
\begin{prop}\label{psi}
The map below  defines a morphism $\psi$ of algebras,  from
 ${\mathcal E}$ to ${\Bbb C}W_n$.
$$
T_i \mapsto s_i, \qquad
E_i \mapsto e_i:=\frac{1}{2}(1 + t_it_{i+1}).
$$
\end{prop}
\begin{proof}
One verifies  that $s_i$'s and $e_i$'s satisfy the relations
 ~\ref{eq1} to ~\ref{eq8} when we put $s_i$ in place of $T_i$,
and $e_i$ to the place of  $E_i$ .
\end{proof}

Thus, from the above proposition the representations
$V_{(\alpha, \beta)}$ are  ${\mathcal E}$-modules.
We are going to prove that  $V_{(\alpha, \beta)}$ and $V_{( \beta, \alpha)}$
are equivalent as  ${\mathcal E}$-modules, and that
$V_{( \alpha, \alpha)}$ is reducible as  ${\mathcal E}$-module.

To do this we need some facts. Firstly,
we will use the Mackey model for describing the induced representations.
 More precisely, let us denote by $V$ an underlying for the
irreducible representation  $\rho$ of a subgroup $H$ of $G$.
Then $\mbox{Ind}_H^G\rho$ can be realized as the
 following vector space
$$
V_{\rho} := \{f: G \longrightarrow V \,
:\, f(hg) = \rho(h)f(g), \, h\in H, \,
g\in G\},
$$
with  action  given by $(gf)(x)= f(xg)$, $g, x\in W_n$.

Let $\{v_i\}$ be a basis of $V$, and $X$ a
set of right coset representatives
of $H$ in $G$. The Dirac basis for $V_{\rho}$ with respect to $X$ and
$\{v_i\}$ is by definition the basis
$\{\delta_{u, i}\, ; u\in X, \, 1\leq i \leq \mbox{dim}V\}$, where
$\delta_{u,i}\in V_{\pi}$ is defined as
$$
\delta_{u, i}(g) =
\left\{\begin{array}{cc}
\rho(h)v_i &\quad\mbox{if} \quad g =hu,\\
0  & \mbox{otherwise}.
\end{array}\right.
$$
(We have $g\delta_{u, i}= \delta_{ug^{-1},i}$.)

Also we use some notations and facts from \cite{dj}.
 Let $a, b\geq 0$ such that $a+b= n$. Let us consider
 the element $w_{a, b}$ in ${\mathcal S}_n$  defined by
 $w_{a, b}= 1$ if $a$ or $b$ is zero. And
$$
w_{a, b} : = (s_{a+b, 1})^b \quad \mbox{if}\quad a,b>0,
$$
where $s_{i,i}= 1$, and
$$
s_{i,j} =
\left\{\begin{array}{ll}
s_{i}s_{i+1,j} &\quad\mbox{if} \quad i<j,\\
s_{i,j+1}s_{j} &\quad\mbox{if} \quad i>j.
\end{array}\right.
$$
We will use  the following properties of the elements $w$'s:
\begin{equation}\label{w1}
w_{a,b}^{-1} = w_{b,a}
\end{equation}
\begin{equation}\label{w2}
w_{a,b}s_k =
\left\{
\begin{array}{cl}
s_{a+k}w_{a,b} & \quad \mbox{if} \quad 1\leq k < b, \\
s_{k-b}w_{a,b} & \quad \mbox{if} \quad b+1\leq k < n.
\end{array}
\right.
\end{equation}
\begin{equation}\label{w3}
{\mathcal S}_{(a,b)}w_{a,b} = w_{a,b}{\mathcal S}_{(b,a)},
\end{equation}
where ${\mathcal S}_{(a, b)}$ is the homomorphic image by $\pi$ of
$W_{(a, b)}$. Thus
${\mathcal S}_{(a, b)}$ is the subgroup of ${\mathcal S}_n$ isomorphic to the
 group ${\mathcal S}_a\times {\mathcal S}_b$.

Let us   denote  by $S_{(a,b)}$ the subset of $S=\{s_1, \ldots , s_{n-1}\}$
  generating the subgroup ${\mathcal S}_{(a,b)}$ of ${\mathcal S}_{n}$.
We have $S_{(a,b)}w_{a,b} = w_{a,b}S_{(b,a)}$.

Let $X_{(a,b)}$ be the set of distinguished  right coset
representatives of ${\mathcal S}_{(a,b)}$ in ${\mathcal S}_n$. Thus
$w_{a,b}\in X_{(a, b)}$, and
\begin{equation}\label{w4}
X_{(a,b)}=w_{a, b} X_{(b,a)}.
\end{equation}

\begin{prop}\label{equiv}
We have $V_{(\alpha, \beta)}\simeq V_{(\beta,\alpha)}$ as
${\mathcal E}$-module.
\end{prop}
\begin{proof}
 Let $\{v_i\}$ (respectively $\{v^{\prime}_j\}$) be a basis for an
underlying $V$ (respectively $V^{\prime}$) of the representation $\alpha$
 (respectively $\beta$) of ${\mathcal S}_a$ (respectively ${\mathcal S}_b$).
Let $\{\delta_{u,(i, j)}\}$ be a Dirac basis for $V_{(\alpha, \beta)}$,
relative to $X_{(a,b)}$ and $\{v_i\otimes v_j\}$. The  maps
$\delta_{u,(i, j)} \mapsto \delta_{w^{-1}u,(j, i)}$
($w=w_{a,b}$)  define a linear isomprphism $\Phi$ from
$V_{(\alpha, \beta)}$ to $V_{(\beta,\alpha)}$. Thus, for the proof of the
proposition we need only to prove:
$$
T_r \circ \Phi = \Phi \circ T_r, \leqno (\ref{equiv}.1)
$$
$$
e_r \circ \Phi = \Phi \circ e_r, \leqno (\ref{equiv}.2)
$$
For any $1\leq r\leq n-1$.

To prove ~\ref{equiv}.1,  we need only to check
$$
(s \circ \Phi)(\delta_{u, (i,j)}) = (\Phi \circ s)(\delta_{u,(i,j)}),
$$
for all $s \in S$. Now, from a
result of V. Deodhar (Lemma 2.1.2 in \cite{gp}), we have either the cases:
$us \in X_{(a,b)}$ or $us= s^{\prime}u$ for some $s^{\prime}\in S_{(a,b)}$.
If we are in the first case the situation is trivial. So, suppose we are in
the second case. Put
$
(\alpha\otimes\epsilon\beta)(s^{\prime}) (v_i\otimes v^{\prime}_j) =
\sum_{k,l}\lambda_{k,l}v_k\otimes v^{\prime}_l,
$
then, it is easy to check that
$
(\beta\otimes\epsilon\alpha)(w^{-1}s^{\prime}w)
(v_j^{\prime}\otimes v_i) =
\sum_{k,l}\lambda_{k,l}v^{\prime}_l\otimes v_k.
$
(notice that $\epsilon (s^{\prime})= 1 $).

Now, we have
$$
s\delta_{u,(i,j)}= \sum_{k,l}\lambda_{k,l}\delta_{u,(k,l)}.
$$
In fact, let $g=hu\in W_n$, $h\in W_{(a,b)}$, we have
$
(s\delta_{u, (i,j)})(g) = \delta_{u, (i,j)}(h(us)) =
\delta_{u, (i,j)}(hs^{\prime}u)$. Thus,
\begin{eqnarray*}
(s\delta_{u, (i,j)})(g) & = &
(\alpha \otimes\epsilon\beta)(hs^{\prime})v_i\otimes v_j^{\prime}\\
& = &
(\alpha \otimes\epsilon\beta)(h)
(\alpha \otimes\epsilon\beta)(s^{\prime})v_i\otimes v_j^{\prime}\\
& = &
(\alpha \otimes\epsilon\beta)(h)\sum_{k,l}\lambda_{k,l}
v_k\otimes v_l^{\prime}\\
&=&
\sum_{k,l}\lambda_{k,l}(\alpha \otimes\epsilon\beta)(h)
v_k\otimes v_l^{\prime}\\
&=&
\sum_{k,l}\lambda_{k,l}\delta_{u,(k,l)}(g).
\end{eqnarray*}
In a similar way we get
$$
s\delta_{w^{-1}u,(j,i)} =
\sum_{k,l}\lambda_{k,l}\delta_{w^{-1}u,(k,l)}
$$

Thus,
$$
\delta_{u, (i,j)}\stackrel{s}{\longmapsto}
\sum_{l,k}\lambda_{k,l}\delta_{u, (k,l)}
 \stackrel{\Phi}{\longmapsto} \sum_{k,l}\lambda_{k,l}\delta_{w^{-1}u, (l,k)}.
$$

On the other hand, we have
$$
\delta_{u, (i,j)}
\stackrel{\Phi}{\longmapsto} \delta_{w^{-1}u, (j,i)}
\stackrel{s}{\longmapsto} s\delta_{w^{-1}u, (j,i)} =
\sum_{k,l}\lambda_{k,l}\delta_{w^{-1}u, (l,k)}.
$$

Hence ~\ref{equiv}.1 follows.

To prove ~\ref{equiv}.2,  we  see that
\begin{equation}\label{titi+1}
t_rt_{r+1}\delta_{u,(i,j)}=
(\alpha \otimes \epsilon\beta )(ut_rt_{r+1}u^{-1})\delta_{u,(i,j)}.
\end{equation}
(notice that $ \alpha \otimes \epsilon\beta$ is a linear character on $C$.)
Then,
\begin{eqnarray*}
(\Phi\circ t_rt_{r+1} )(\delta_{u,(i,j)})&=&
(\alpha \otimes \epsilon\beta )(ut_rt_{r+1}u^{-1})\delta_{w^{-1}u,(i,j)},\\
(t_rt_{r+1}\circ\Phi)(\delta_{u,(i,j)})&=&
(\beta \otimes \epsilon\alpha )(w^{-1}ut_rt_{r+1}u^{-1}w)
\delta_{w^{-1}u,(i,j)}.
\end{eqnarray*}
Therefore we deduce ~\ref{equiv}.2, because $t_it_{i+1}$ is a diagonal
matrix, and
$$
w^{-1}diag(d_1, d_2 \ldots, d_a,d_{a+1}, \ldots , d_n)w =
diag(d_{a+1}, \ldots , d_n, d_1, d_2, \ldots d_a).
$$
\end{proof}
 From eq.(\ref{titi+1}) we get:
\begin{cor}\label{e} For any $\delta$  in a Dirac basis, we have
$$
e_r\delta =
\left\{\begin{array}{cl}
1 &\quad\mbox{if} \quad r\not=a,\\
0 &\quad\mbox{if} \quad r=a.
\end{array}\right..
$$
\end{cor}

Now, we   decompose the ${\mathcal E}$-module
$V_{(\alpha, \alpha)}$ into two modules. Note that this situation appears only
when $n$ is an even number. Set $m = n/2$, and put  $S^{\prime}=  S-\{s_m\}$.
 Thus  ${\mathcal S}_{(m, m)}$  is generated by $S^{\prime}$.

Let $w$ be the element $w_{m,m}$, and let $X$ be the set of distinguished
elements $X_{(m, m)}$.  From ~\ref{w4} we have $X = wX$. Then we can  choose a
subset $Y$ of $X$ such that $X = Y \cup wY$ (disjoint union).
Let $\{\delta_{u,i}\}$ be the Dirac basis of $V_{(\alpha, \alpha)}$
with respect to $X$ and  the basis $\{v_i\}$ of  $V\otimes V$,
where $V$ is an underlying for the representation $\alpha$ of
${\mathcal S}_m$. We have

\begin{prop}\label{descom}
 $V_{(\alpha, \alpha)} = V_{\alpha}^+ \oplus V_{\alpha}^-$
as ${\mathcal E}$-module, where
$$
V_{\alpha}^+ := \langle \delta_{u,i} + \delta_{wu,i}\, ; \,
u \in Y, \, 1\leq i\leq (\mbox{dim}V)^2 \rangle ,
$$
$$
V_{\alpha}^- := \langle \delta_{u,i} - \delta_{wu,i}\, ; \,
u \in Y, \, 1\leq i\leq (\mbox{dim}V)^2 \rangle.
$$
\end{prop}
\begin{proof} The decomposition of  vector space  is obvious.
So, we  need only to prove that $V_{\alpha}^+$ (respectively
$V_{\alpha}^-$) is a
${\mathcal E}$- module. For this,  it is sufficient to prove:
$$
s(\delta_{u,i} + \delta_{wu,i}) \in V_{\alpha}^+ \quad (s\in S),
\leqno (\ref{descom}.1)
$$
$$
e_r (\delta_{u,i} + \delta_{wu,i}) \in V_{\alpha}^+.
\leqno (\ref{descom}.2)
$$

Let us see ~\ref{descom}.1. For any $s \in S$, we have
$s\delta_{u,i} = \delta_{us,i}$. Again from  V. Deodhar's result, we have
either the cases:  $us \in X$
or $us=s^{\prime}u$, for some $s^{\prime}\in S^{\prime}$. Suppose we are
in the first case. If $us=u^{\prime}\in Y$, the situation
is trivial,  if $us= u^{\prime}\in wY$, hence $wu^{\prime}\in Y$, then
$s(\delta_{u,i} + \delta_{wu,i}) =
\delta_{wu^{\prime},i}+\delta_{u^{\prime},i}\in V_{\alpha}^+$.

If we are in  the second case, put
$us = s^{\prime}u$, with $s^{\prime}\in S^{\prime}$. Set
$$
(\alpha \otimes \epsilon\alpha) (s^{\prime})v_i = \sum_{j}\lambda_{i,j}v_j.
$$
We claim that
$$
(a)\,\, s\delta_{u, i} = \sum_{j}\lambda_{i,j}\delta_{u,j}, \quad \mbox{and}
 \quad (b)\,\, s\delta_{wu, i} = \sum_{j}\lambda_{i,j}\delta_{wu,j}.
$$
Then we
have obtained what we were looking for,
since
$$
s(\delta_{u,i}  +\delta_{wu, i}) = \sum_{j}\lambda_{i,j}(\delta_{u,i}+
\delta_{wu,i}).
$$

Now,  claim (a)  is easy to check. Let us prove    claim (b).
Let $x= h(wu)\in W_n$, with $h\in {\mathcal S}_{(m, m)}$, we have
$xs = h(wu)s =hws^{\prime}u$. Then
\begin{eqnarray*}
(s\delta_{wu,i})(x)& = & \delta_{wu,i}(hws^{\prime}u)\\
                   & = & (\alpha \otimes \epsilon\alpha)
                   (hws^{\prime}w)(v_i) \qquad (hws^{\prime}u =
                   (hws^{\prime}w)wu)\\
                   & = & (\alpha \otimes \epsilon\alpha)(h)
                   (\alpha \otimes \epsilon\alpha)(ws^{\prime}w)(v_i) \\
                   & = & (\alpha \otimes \epsilon\alpha)(h)
                   (\alpha \otimes \epsilon\alpha)(s^{\prime})(v_i)
                   \qquad (\mbox{from (\ref{w3})})\\
                   & = & (\alpha \otimes \epsilon\alpha)(h)
                   \sum_{j}\lambda_{i,j}v_j\\
                   & = & \sum_{j}\lambda_{i,j}
                   (\alpha \otimes \epsilon\alpha)(h)v_j\\
                   & = & \sum_{j}\lambda_{i,j}\delta_{wu,j}v_j.
\end{eqnarray*}
Then $s\delta_{wu, i} = \sum_{j}\lambda_{i,j}\delta_{wu,j}$.

Now ~\ref{descom}.2 follows from (\ref{e}).

Similarly one can prove that $V_{\alpha}^-$ is a ${\mathcal E}$-module.
\end{proof}

Now, we have two homomorphisms $\varphi_0$ and $\varphi_{1}$
 from ${\mathcal E}$ to the algebra ${\Bbb C}{\mathcal S}_n$:
$\varphi_0$ is defined by sending $T_i$ to $s_i$, and $E_i$ to $0$, and
$\varphi_1$ is defined by sending $T_i$ to $s_i$, and $E_i$ to $1$. Then,
these morphisms yield two families of irreducible representations
non-equivalent of ${\mathcal E}$. Set $\{(\alpha, 0)\}$ (respectively
$\{(\alpha, 1)\}$) the family of
irreducible representations  yielded by $\varphi_0$ (respectively
$\varphi_1$).
\begin{prop}\label{esen}
We have the following equivalence of ${\mathcal E}$-modules:
$$
(\alpha, 1)\simeq V_{(\alpha, \phi)}\qquad (\mbox{$\alpha$  partition of $n$,
$\forall n$}).
$$
\end{prop}
\begin{proof} The equivalence follows from the fact that
$e_i$ acts trivially on the Delta basis. See Corollary ~\ref{e}.
\end{proof}

We shall denote the  found representation by \lq\lq Young diagram\rq\rq ,
 which has encoded  information of the dimension of the representation.
 The representations of Proposition ~\ref{equiv} are denoted by
$(\alpha, \beta)$, and the representations $V_{\alpha}^+$
(respectively $V_{\alpha}^-$ ) of the Proposition ~\ref{descom} are denoted
by $(\alpha, +)$ (respectively $(\alpha, -)$).

In the case $n=2$ the algebra ${\mathcal E}$ is of dimension 4, and we have
 four  non-equivalent representations of dimension 1 for ${\mathcal E}$:
$(
\setlength\unitlength{0.3ex}
\begin{picture}(10,5)
\put(0,0){\line(1,0){5}}
\put(0,5){\line(1,0){5}}
\put(0,0){\line(0,1){5}}
\put(5,0){\line(0,1){5}}

\put(5,0){\line(1,0){5}}
\put(5,5){\line(1,0){5}}
\put(5,0){\line(0,1){5}}
\put(10,0){\line(0,1){5}}
\end{picture}, \phi)$,
$(
\setlength\unitlength{0.3ex}
\begin{picture}(5,5)
\put(0,-2){\line(1,0){5}}
\put(0,3){\line(1,0){5}}
\put(0,-2){\line(0,1){5}}
\put(5,-2){\line(0,1){5}}

\put(0,8){\line(1,0){5}}
\put(0,8){\line(1,0){5}}
\put(0,3){\line(0,1){5}}
\put(5,3){\line(0,1){5}}
\end{picture}, \phi )$,
$(
\setlength\unitlength{0.3ex}
\begin{picture}(5,5)
\put(0,0){\line(1,0){5}}
\put(0,5){\line(1,0){5}}
\put(0,0){\line(0,1){5}}
\put(5,0){\line(0,1){5}}
\end{picture}, +)$,
$(
\setlength\unitlength{0.3ex}
\begin{picture}(5,5)
\put(0,0){\line(1,0){5}}
\put(0,5){\line(1,0){5}}
\put(0,0){\line(0,1){5}}
\put(5,0){\line(0,1){5}}
\end{picture}, -).
$

In the case $n=3$ the irreducible representations of ${\mathcal E}$ are:

\begin{eqnarray*}
 & \mbox{dimension} \\
(
\setlength\unitlength{0.3ex}
\begin{picture}(15,5)
\put(0,0){\line(1,0){5}}
\put(0,5){\line(1,0){5}}
\put(0,0){\line(0,1){5}}
\put(5,0){\line(0,1){5}}

\put(5,0){\line(1,0){5}}
\put(5,5){\line(1,0){5}}
\put(5,0){\line(0,1){5}}
\put(10,0){\line(0,1){5}}

\put(10,0){\line(1,0){5}}
\put(10,5){\line(1,0){5}}
\put(10,0){\line(0,1){5}}
\put(15,0){\line(0,1){5}}
\end{picture}, \phi) & 1 \\
(
\setlength\unitlength{0.3ex}
\begin{picture}(5,10)
\put(0,0){\line(1,0){5}}
\put(0,5){\line(1,0){5}}
\put(0,0){\line(0,1){5}}
\put(5,0){\line(0,1){5}}

\put(0,5){\line(1,0){5}}
\put(0,10){\line(1,0){5}}
\put(0,5){\line(0,1){5}}
\put(5,5){\line(0,1){5}}

\put(0,-5){\line(1,0){5}}
\put(0,0){\line(1,0){5}}
\put(0,-5){\line(0,1){5}}
\put(5,-5){\line(0,1){5}}
\end{picture}, \phi) & 1 \\
(
\setlength\unitlength{0.3ex}
\begin{picture}(10,5)
\put(0,0){\line(1,0){5}}
\put(0,5){\line(1,0){5}}
\put(0,0){\line(0,1){5}}
\put(5,0){\line(0,1){5}}

\put(5,0){\line(1,0){5}}
\put(5,5){\line(1,0){5}}
\put(5,0){\line(0,1){5}}
\put(10,0){\line(0,1){5}}

\put(0,-5){\line(1,0){5}}
\put(0,0){\line(1,0){5}}
\put(0,-5){\line(0,1){5}}
\put(5,-5){\line(0,1){5}}
\end{picture}, \phi ) & 2 \\
(
\setlength\unitlength{0.3ex}
\begin{picture}(15,5)
\put(0,0){\line(1,0){5}}
\put(0,5){\line(1,0){5}}
\put(0,0){\line(0,1){5}}
\put(5,0){\line(0,1){5}}

\put(5,0){\line(1,0){5}}
\put(5,5){\line(1,0){5}}
\put(5,0){\line(0,1){5}}
\put(10,0){\line(0,1){5}}

\put(10,0){\line(1,0){5}}
\put(10,5){\line(1,0){5}}
\put(10,0){\line(0,1){5}}
\put(15,0){\line(0,1){5}}
\end{picture}, 0) & 1  \\
(
\setlength\unitlength{0.3ex}
\begin{picture}(5,10)
\put(0,0){\line(1,0){5}}
\put(0,5){\line(1,0){5}}
\put(0,0){\line(0,1){5}}
\put(5,0){\line(0,1){5}}

\put(0,5){\line(1,0){5}}
\put(0,10){\line(1,0){5}}
\put(0,5){\line(0,1){5}}
\put(5,5){\line(0,1){5}}

\put(0,-5){\line(1,0){5}}
\put(0,0){\line(1,0){5}}
\put(0,-5){\line(0,1){5}}
\put(5,-5){\line(0,1){5}}
\end{picture}, 0) & 1 \\
(
\setlength\unitlength{0.3ex}
\begin{picture}(10,5)
\put(0,0){\line(1,0){5}}
\put(0,5){\line(1,0){5}}
\put(0,0){\line(0,1){5}}
\put(5,0){\line(0,1){5}}

\put(5,0){\line(1,0){5}}
\put(5,5){\line(1,0){5}}
\put(5,0){\line(0,1){5}}
\put(10,0){\line(0,1){5}}

\put(0,-5){\line(1,0){5}}
\put(0,0){\line(1,0){5}}
\put(0,-5){\line(0,1){5}}
\put(5,-5){\line(0,1){5}}
\end{picture}, 0 ) & 2 \\
(
\setlength\unitlength{0.3ex}
\begin{picture}(10,5)
\put(0,0){\line(1,0){5}}
\put(0,5){\line(1,0){5}}
\put(0,0){\line(0,1){5}}
\put(5,0){\line(0,1){5}}

\put(5,0){\line(1,0){5}}
\put(5,5){\line(1,0){5}}
\put(5,0){\line(0,1){5}}
\put(10,0){\line(0,1){5}}
\end{picture},
\setlength\unitlength{0.3ex}
\begin{picture}(10,5)
\put(0,0){\line(1,0){5}}
\put(0,5){\line(1,0){5}}
\put(0,0){\line(0,1){5}}
\put(5,0){\line(0,1){5}}
\end{picture}) & 3 \\
(
\setlength\unitlength{0.3ex}
\begin{picture}(5,5)
\put(0,-2){\line(1,0){5}}
\put(0,3){\line(1,0){5}}
\put(0,-2){\line(0,1){5}}
\put(5,-2){\line(0,1){5}}

\put(0,8){\line(1,0){5}}
\put(0,8){\line(1,0){5}}
\put(0,3){\line(0,1){5}}
\put(5,3){\line(0,1){5}}
\end{picture},
\setlength\unitlength{0.3ex}
\begin{picture}(10,5)
\put(0,0){\line(1,0){5}}
\put(0,5){\line(1,0){5}}
\put(0,0){\line(0,1){5}}
\put(5,0){\line(0,1){5}}
\end{picture}) & 3
\end{eqnarray*}

\begin{thm}\label{ss}
The algebra ${\mathcal E}_3(1)$ is semisimple, and we have
$$
{\mathcal E}_3(1) = 4M_1({\Bbb C})\oplus 2M_2({\Bbb C})\oplus
2M_3({\Bbb C}).
$$
\end{thm}
\begin{proof} We shall prove that the dimension of
${\mathcal E}= {\mathcal E}_3(1)$ is 30. Then, the theorem follows.

Let us consider the linear generators for ${\mathcal E}$
of Corollary ~\ref{span}:
$$
\{T_w, T_wE_1, T_wE_2, T_wE_1E_2, T_wE_2T_1\, ;\, w\in {\mathcal S}_n\}.
$$
We must  prove that if:
$$
\sum_{w\in {\mathcal S}_3}a_wT_w +
\sum_{w\in {\mathcal S}_3}A_w T_wE_1 +
\sum_{w\in {\mathcal S}_3}B_w T_wE_2 +
\sum_{w\in {\mathcal S}_3}C_w T_wE_1E_2 +
\sum_{w\in {\mathcal S}_3}D_wT_wE_2T_1 = 0, \leqno (\ref{ss}.1)
$$
then $a_w=A_w=B_w=C_w=D_w= 0, \quad \forall w$.

Applying $\varphi_0$  to equation ~\ref{ss}.1, we deduce that
$a_{w}=0$, for all $w$.

Now, applying the morphism $\psi$ of Proposition ~\ref{psi}
 to ~\ref{ss}.1, we get the equation
$$
\sum_{w\in {\mathcal S}_3}A_w we_1 +
\sum_{w\in {\mathcal S}_3}B_w we_2 +
\sum_{w\in {\mathcal S}_3}C_w we_1e_2 +
\sum_{w\in {\mathcal S}_3}D_w we_3    = 0, \leqno (\ref{ss}.2)
$$
where $e_3 := s_2e_1 = \frac{1}{2}(1 + t_1t_3)$.

The left part of ~\ref{ss}.2 can be written in terms of elements of
${\mathcal D}:= \{w, wt_1t_2, wt_2t_3, wt_1t_3\}$,
which is a subset of the canonical basis
$\{wt \,; \, w\in {\mathcal S}_n, t\in C \}$ of ${\Bbb C}W_n$. Hence
${\mathcal D}$ is an  independent linear set.
Thus, one can  deduce that $A_w=B_w=C_w=D_w= 0$,  for all $w$, studying
the coefficient of the elements of ${\mathcal D}$. More precisely,

The coefficient  of 1 in  ~\ref{ss}.2 is $ A_1 + B_1 + C_1/2 + D_1$,

The coefficient  of $t_1t_2$ in  ~\ref{ss}.2 is $ A_1 + C_1/2$,

The coefficient  of $t_2t_3$ in  ~\ref{ss}.2 is $ B_1 + C_1/2$,

The coefficient  of $t_1t_3$ in  ~\ref{ss}.2 is $ C_1/2 + D_1$.

As all these coefficients are 0, we get $A_1 = B_1= C_1=D_1=0$.

Finally, multiplying ~\ref{ss}.2 for convenient $w$ one can deduce
$A_w=B_w=C_w=D_w= 0$, $\forall w$.
\begin{rem}
From the proof of Theorem ~\ref{ss}, we deduce that $\psi$ is one-to- one.
\end{rem}

\end{proof}


\begin{thebibliography}{ }


\bibitem{AJ} F. Aicardi, J. Juyumaya, {\em An algebra involving braids and ties.} Preprint ICTP   IC/2000/179,
Trieste.
\bibitem{aijuMMJ1} F. Aicardi, J. Juyumaya, {\it Markov trace on the algebra of braid and ties}, Moscow
Math. J. 16(3) (2016) 397--–431.

\bibitem{aijuJKTR1} F. Aicardi, J. Juyumaya, {\it Tied   Links},   J. Knot Theory Ramifications, \textbf{ 25} (2016), no. 9,   DOI: 10.1142/S02182165164100171.

\bibitem{aijuMZ} F. Aicardi, J. Juyumaya, {\it Kauffman type  invariants for  tied links},  to be published in  Mathematische Zeitschrift.


\bibitem{rh} S.  Ryom-Hansen, {\em On the representation theory of an algebra of braids and ties.} J. Algebr. Comb.
33   (2011),  57--79.

\bibitem{eb} Elizabeth O. Banjo, {\em The Generic Representation Theory of the Juyumaya
Algebra of Braids and Ties} Algebr. Represent. Theor. 16 (2013), 1385 --1395
\bibitem{dj} R. Dipper and  G. James, {\em Representations of Hecke algebra of
type $B_n$}, J. Algebra,  {\bf 146},  (1992),  454-461.
\bibitem{gp} M. Geck and G. Pfeifer, {\em Characters of finite Coxeter groups
and Iwahori-Hecke algebras}, Oxford Sciencies Publ., 2000.
\bibitem{jk} J. Juyumaya,  S. Senthamarai Kannan {\em Braid relations in
the Yokonuma-Hecke algebra} J. Algebra 239 , no. 1 (2001), 272--297.
\bibitem{ju} J. Juyumaya, {\em Another algebra from the Yokonuma-Hecke
algebra}. Preprint ICTP, IC/1999/160.
\bibitem{bw} J. Birman and H. Wenzl, {\em Braids, link polynomials,
and a new algebra}, Trans. Amer. Math. Soc. {\bf 313}  (1989), 249-273.
\end{thebibliography}
\end{document}